\documentclass[11pt]{amsart}

\usepackage{amsfonts,amsmath,latexsym,amssymb,verbatim,amsbsy,amsthm}
\usepackage{dsfont,bm}

\usepackage{fancyhdr}
\usepackage{epsfig}
\usepackage{setspace}
\usepackage{dsfont}
\usepackage{listings}
\usepackage{float}
\usepackage{caption}
\usepackage{subcaption}

\usepackage[top=1in, bottom=1in, left=1in, right=1in]{geometry}

\usepackage[dvipsnames]{xcolor}

\usepackage[colorlinks=true, pdfstartview=FitV, linkcolor=RoyalBlue,citecolor=ForestGreen, urlcolor=blue]{hyperref}

\definecolor{labelkey}{rgb}{0,0,1}

\theoremstyle{plain}
\newtheorem{THEOREM}{Theorem}[section]

\newtheorem{theorem}[THEOREM]{Theorem}

\newtheorem{lemma}[THEOREM]{Lemma}

\theoremstyle{definition}

\newtheorem{definition}[THEOREM]{Definition}

\theoremstyle{remark}

\newcommand{\thm}[1]{Theorem~\ref{#1}}
\newcommand{\lem}[1]{Lemma~\ref{#1}}

\newcommand{\sect}[1]{Section~\ref{#1}}

\def \a {\alpha}

\def \d {\delta}
\def \e {\varepsilon}
\def \f {\varphi}

\def \n {\nabla}
\def \s {\sigma}

\def \th {\theta}

\def \Th {\Theta}
\def \O {\Omega}

\def \by {{\bf y}}

\def \cD {\mathcal{D}}
\def \cE {\mathcal{E}}

\def \cP {\mathcal{P}}

\def \cW {\mathcal{W}}

\newcommand{\N}{\ensuremath{\mathbb{N}}}   
\newcommand{\R}{\ensuremath{\mathbb{R}}}   

\renewcommand{\geq}{\geqslant}

\renewcommand{\leq}{\leqslant}

\DeclareMathOperator{\supp}{supp} %
\def \p {\partial}
\def \ra {\rightarrow}
\def \ss {\subset}

\renewcommand{\geq}{\geqslant}

\renewcommand{\leq}{\leqslant}

\def \ds  {\, \mbox{d}s}

\def \dmu  {\, \mbox{d}\mu}
\def \dnu  {\, \mbox{d}\nu}

\def \ddt  {\frac{\mbox{d\,\,}}{\mbox{d}t}}

\def \dpi {\, \mbox{d} \pi}
\def \dtpi {\, \mbox{d} \tpi}

\def \dtmu {\, \mbox{d} \tmu}

\def \Lip {\mathrm{Lip}}

\def \tpi {\tilde{\pi}}

\def \tmu {\tilde{\mu}}

\def \thmax {\th_{\mathrm{max}}}
\def \thmin {\th_{\mathrm{min}}}

\def \Ymax {Y_{\mathrm{max}}}
\def \Ymin {Y_{\mathrm{min}}}

\def \ymax {y_{\mathrm{max}}}
\def \ymin {y_{\mathrm{min}}}

\def \gmax {g_{\mathrm{max}}}
\def \gmin {g_{\mathrm{min}}}

\begin{document}

\title[Continuous model of opinion dynamics with convictions]{Continuous model of opinion dynamics with convictions}

\author{Vinh Nguyen}

\author{Roman Shvydkoy}

\address{Department of Mathematics, Statistics and Computer Science, University of Illinois at Chicago, 60607}

\email{vnguye66@uic.edu }

\email{shvydkoy@uic.edu}

\begin{abstract}
    In this note we study a new kinetic model of opinion dynamics. The model incorporates two forces -- alignment of opinions under all-to-all communication driving the system to a consensus, and Rayleigh type friction force that drives each `player' to its fixed conviction value. The balance between these forces creates a non-trivial limiting outcome.  

We establish existence of a global mono-opinion state, whereby any initial distribution of opinions for each conviction value aggregates to the Dirac measure concentrated on a single opinion. We identify several cases where such a state is unique and depends continuously on the initial distribution of convictions. Several regularity properties of the limiting distribution of opinions are presented.
\end{abstract}


\thanks{\textbf{Acknowledgment.}  
	This work was  supported in part by NSF
	grant  DMS-2107956.}

\date{\today}
\maketitle

\section{Introduction}

In this note we study regularity and long time behavior of solutions to the following transport equation
\begin{equation}\label{e:opinion}
\p_t \mu + \p_y(u(\mu) \mu ) = 0,
\end{equation}
where $\mu = \mu(t,y,\th)$ is a measure on $\O = \R_+ \times \R_+$ for each $t\geq 0$, and 
\begin{align}\label{ }
u(\mu) & = \p_y (W\ast \mu + \s V), \\
W(y) & = - \frac12 y^2, \quad V(y,\th)  = \frac12 \th y^2 - \frac{1}{p+2} y^{p+2}.
\end{align}
Here, $\s$ and $p$ are positive parameters. The variable $\th$ can be thought of as a parameter as well, however, note that the convolution $W \ast \mu$ couples all the measures together across the family. 

The motivation for this particular model is twofold. First, it represents the kinetic counterpart  of the corresponding discrete dynamical system:
\begin{equation}\label{e:opiniondisc}
 \dot{y}_i = \frac1N \sum_{k=1}^N  (y_k - y_i ) + \s (\th_i - y_i^p) y_i,
\end{equation}
where $\th_i$'s are constant parameters.  In fact, the empirical distributions 
\[
\mu^N = \frac1N \sum_{i=1}^N \d_{\th_i} \otimes \d_{y_i(t)}
\]
solve \eqref{e:opinion} in the weak sense if and only if $y_i$'s solve \eqref{e:opiniondisc}, and formally the mean-field limit $\mu^N \to \mu$ yields a solution to \eqref{e:opinion}. The discrete system \eqref{e:opiniondisc} was derived in \cite{LRS-friction} as the effective limiting dynamics of the speeds $y_i = |v_i|$ of  agents governed by the corresponding alignment model with all-to-all communication
\begin{equation}\label{e:CSR}
\dot{x}_i = v_i,\qquad \dot{v}_i = \frac1N \sum_{k=1}^N  (v_k - v_i ) + \s (\th_i - |v_i|^p) v_i.
\end{equation}
When all velocities $v_i$ belong to a sector of opening less than $\pi$, the vectors $v_i$ will dynamically align themselves along one direction $v_i \sim y_i \hat{v}$, where $y_i = |v_i|$, and the evolution of $y_i$ is governed by \eqref{e:opiniondisc} up to an exponentially decaying force. The system \eqref{e:CSR} is a very important example of a collective behavior model of Cucker-Smale type with Rayleigh friction that appeared in many previous studies, \cite{Ha-friction,LRS-friction,DOrsogna,Sbook}.

Second, like its discrete counterpart, the equation \eqref{e:opinion} can be viewed as a mean-field model of opinion dynamics. Unlike many models existing in the literature, see \cite{KH-opinion,G2022,BVI2022} and literature therein, this particular version takes into account not only players'  opinions $y$'s, which are flexible quantities, but also convictions $\th$'s, which are fixed and not changing in time. The end game is to find an `agreement' which represents the limiting state of opinions. Such an agreement is not expected to be a perfect consensus due to attraction of each player it its own conviction but it must be one where further change would not benefit any of the players. In \cite{LRS-friction} it was proved that such a state exists, is unique and stable, and represents the Nash equilibrium relative to properly defined payoff functions. Let us state the result.

\begin{theorem}\label{t:opinion} For any positive set of parameters $(\th_1,\ldots,\th_N,\s)\in  \R^N_+ \times \R_+$ there exists a unique stable Nash equilibrium $\by^*=(y_1^*,\ldots,y_N^*) \in \R^N_+$ of system \eqref{e:opiniondisc} relative to payoffs
\begin{equation}\label{e:payoff}
p_i(\by) = \s \left( \frac12 \th_i y_i^2 - \frac{1}{p+2} y_i^{p+2}\right) - \frac{1}{2} \left( \bar{y} - y_i\right)^2, \quad \bar{y} = \frac1N \sum_j y_j.
\end{equation}
Any solutions with positive initial data will remain positive and converge to $\by^*$ as $t \to \infty$. Moreover, if $\th_i = \th_j$ then $y_i = y_j$.
\end{theorem}

The main difficulty in establishing the result is that the natural gradient structure of \eqref{e:opiniondisc}
\[
\dot{\by} = - \n \Phi(\by) 
\]
involves energy $\Phi(\by) = \sum_{i=1}^N p_i(\by)$ that is not globally convex. The approach of  \cite{LRS-friction} is based upon the Brouwer topological degree theory for the existence and uniqueness and the Lojasiewicz-Simon inequality for the asymptotic behavior.

The purpose of this present study is to recreate a similar result for the kinetic model \eqref{e:opinion}.  First, we justify it as the mean-field model of \eqref{e:opiniondisc} by establishing the limit $\mu^N \to \mu$. This is done by proving a general weak-Lipschitzness of the solution map $\mu_0 \to \mu_t$ with respect to the Wasserstein-1 metric, \cite{Ambrosio-book},
\[
\cW_1(\mu_t,\nu_t) \leq C e^{ct}\cW_1(\mu_0,\nu_0), \quad t>0,
\]
see \sect{s:mfl}. It also shows that there exists a global measure-valued solution to \eqref{e:opinion} for any initial data compactly supported inside $\O$.

To state the main result let us fix some notation. Let us observe that the  $\th$-marginal given by
\begin{equation}\label{e:pi}
\dpi(\th) = \int_{y\in \R_+} \dmu(y,\th,t),
\end{equation}
is conserved $\ddt \pi= 0$. This is a reflection of the principle that convictions do not change.  By the disintegration theorem,  see \cite{Ambrosio-book}, for $\pi$-a.e. $\th \in \R_+$ there is a unique family of probability `slicing' measures $\{\mu^\th\}$ on $\R_+$ such that $\mu = \mu^\th\otimes \dpi$, that is, 
\begin{equation}\label{slicing-ms}
\int_{\O} \f(y,\th) \dmu(y,\th) =\int_{\R_+}\int_{\R_+}   \f(y,\th) \dmu^\th(y)\dpi(\th), \quad \forall \f \in C_0(\O).
\end{equation}
Each measure $\mu^\th$ represents a distribution of opinions of agents having the same conviction $\th$.

Out main result states that each of these slicing measures approaches a mono-opinion state, i.e. concentrates on a Dirac mass at a fixed point $g(\th)$ for some smooth strictly increasing function $g$. In other words,
\[
\mu_t \to  \d_{g(\th)} \otimes \dpi(\th).
\]

 To put it formally we assume that our initial measure is located within a box compactly inside $\O$:
\begin{equation}\label{e:R0}
\supp \mu_0 \ss R_0: =  [\ymin,\ymax]\times[\thmin,\thmax], \qquad \ymin, \thmin>0.
\end{equation}
\begin{theorem}\label{t:main}
Let $\mu$ be the measure-valued solution to \eqref{e:opinion} with initial data satisfying \eqref{e:R0}. Then there exists a function $g\in C^\infty([\thmin,\thmax])$ strictly increasing such that
\begin{equation}\label{exp-th}
\sup_{\th \in \supp \pi}\cW_1 (\mu_t^\th, \d_{g(\th)}) \leq  C e^{-c t}, \quad t>0,
\end{equation}
where $C,c>0$ depend only on $\mu_0$ and the parameters of the model. Moreover, under the assumption
\begin{equation}\label{e:!}
\s \thmin > \frac{p+1}{p} \quad \text{or} \quad \frac{\thmax}{\thmin} < p+1.
\end{equation}
 the map $\pi \to g$ is Lipschitz,
\begin{equation}\label{e:gLip}
\sup_{\th \in [\thmin,\thmax] }|g(\th)-\tilde{g}(\th) |  \leq C \cW_1(\pi,\tpi).
\end{equation}
In particular $g$ is unique for each $\pi$.
\end{theorem}

The proof of this theorem is a consequence of our results stated in Lemmas~\ref{l:mono}, \ref{l:unique}, and \ref{l:stability}.

 Structurally, the equation \eqref{e:opinion} can be considered as a fibered gradient system in the sense of \cite{PP2022} where the fibers are parametrized by convictions $\th$ and the free energy is given by
\begin{equation}\label{ }
\cE(\mu) =  \frac12 \int_{\R_+^2 \times \R_+^2} W(x-y) \dmu(y,\th) \dmu(x,\eta) - \s \int_{\R^2_+} V(y,\th) \dmu(y,\th).
\end{equation}
The equation can be written as a gradient dynamics
\[
\p_t \mu = - \p \cE(\mu),
\]
where $\p$ is understood as a fibered variant of the Fr\'echet subdifferential relative to a properly defined fibered Wasserstein distance. Without getting further into details one can obtain directly the following energy dissipation law
\begin{equation}\label{ }
\ddt \cE = - \int_{\R_+ \times \R_+}| u(\mu) |^2 \dmu(y,\th).
\end{equation}
The law demonstrates perpetual descent of the energy along solutions and suggests convergence to a local minimum. The general results of this nature were established in \cite{PP2022} under a properly formulated convexity condition on the energy. However, just as in the discrete case, such convexity is not always true in our settings, and therefore the statement of \thm{t:main} does not directly from the theory developed in \cite{PP2022}.
The novelty of our method is in detailed analysis of asymptotic behavior of characteristics of equation \eqref{e:opinion}.  Let us note that in the discrete case the uniqueness of the limiting state is unconditional, the fact that we have not been able to establish yet for kinetic counterpart. We leave this for future research.

\section{Well-posedness and mean-field limit}\label{s:mfl}
In this section, we will prove the existence of measure-valued solutions to the equation \eqref{e:opinion}. First of all, let us introduce some notations and definitions. Let $\O = \R_+^2$ and denote $\cP_0(\O)$ the set of probability measures on $\O$ which have compact support in the interior of $\O$.
\begin{definition}\label{ms-soln}
Given $0\leq T <\infty$, a map $\mu: [0,T) \ra \cP_0(\O), \; t\mapsto \mu_t$, is called a measure-valued solution to \eqref{e:opinion} with initial data $\mu_0$ if it satisfies the following conditions:
\begin{itemize}
    \item[i)] $\mu$ is weakly* continuous, 
\item[ii)] For any $\f \in C^{\infty}_0([0,T)\times \O)$ and $0<t<T$,
\[
\begin{split}
\int_{\O} \f(t,y,\th)\dmu_t(y,\th) =     \int_{\O}\f(0,y,\th) \dmu_0(y,\th) + \int_0^t \int_{\O} [\p_s\f +  u\,\p_y\f]\dmu_s(y,\th)\ds.
\end{split}
\]
\end{itemize}
\end{definition}

To make further notation simpler let us observe that by making the change of variables
\begin{equation}\label{e:change}
y \to \s^{\frac1p} y, \qquad \th \to \s \th, \qquad \mu \to \s^{1+ \frac1p} \mu
\end{equation}
we can scale out the parameter $\s$ from the equation altogether.  So, from now on we can assume that $\s = 1$, and be mindful that all the constants that appear later eventually depend on the original parameter $\s$.

If $\mu: [0,T) \ra \cP_0(\O)$ is a measure-valued solution to \eqref{e:opinion} with initial data $\mu_0$, by the classical transport theory, $\mu$ is a push-forward of $\mu_0$ along characteristics $(Y, \Th)$:  
\begin{align}
    \ddt Y(t,y,\th) &= \int_{\O} (Y'- Y)\dmu_0(y',\th') +  Y(\Th - Y^p),\quad Y(0,y,\th) = y, \label{e:Y}\\
        \ddt \Th(t,y,\th) &= 0, \quad \Th(0,y,\th) =\th\label{e:Th}.
\end{align}
Note that $\Th$ is not changing in time, so in the equation \eqref{e:Y} we can replace $\Th$ by its initial $\th$ and view $\th$ as a parameter.

The local well-posedness of the system \eqref{e:Y} - \eqref{e:Th} follows from the standard fixed point argument for integro-differential equations and local Lipschitzness relative to continuous maps $(Y,\Th)$    of the right hand side. Global well-posedness will follow as soon as we establish a priori bounds on the support of $Y$.

Our standing assumption on the initial support of $\mu_0$ will always be \eqref{e:R0}.  Let us denote
\[
\Ymax = \max_{R_0} Y(t,\cdot), \quad \Ymin = \min_{R_0} Y(t,\cdot).
\]
Note  that $\ymax = \Ymax(0)$ and $\ymin = \Ymin(0)$.
\begin{lemma}\label{ }
For any solution $Y$ to \eqref{e:Y} on a time interval $[0,T)$, we have for all $t<T$,
\begin{align}
\Ymax^p & \leq \frac{\thmax \ymax^p e^{p \thmax t}}{\thmax +  \ymax^p (e^{p \thmax t} - 1)} \label{e:Ymax}\\
\Ymin^p & \geq \frac{\thmin \ymin^p e^{p \thmin t}}{\thmin +  \ymin^p (e^{p \thmin t} - 1)} \label{e:Ymin}
\end{align}
\end{lemma}
\begin{proof}
Evaluating  \eqref{e:Y} at a point of maximum on $R_0$ we obtain
\[
\begin{split}
    \ddt \Ymax^p &= p \Ymax^{p-1}\underbrace{  \int_{\O} (Y'- \Ymax)\dmu_0(y',\th')}_{\leq 0} +  \Ymax^p (\th - \Ymax^p)\\
    &\leq  p \Ymax^p( \thmax -\Ymax^p).
\end{split}
\]
The right hand side of \eqref{e:Ymax} solves the above equation exactly. So, by the classical comparison principle, we obtain \eqref{e:Ymax}. 

Similarly,
\[
\begin{split}
    \ddt \Ymin^p &= p \Ymin^{p-1}\underbrace{  \int_{\O} (Y'- \Ymin)\dmu_0(y',\th')}_{\geq 0} +  \Ymin^p (\th - \Ymin^p)\\
    &\geq  p \Ymin^p( \thmin -\Ymin^p).
\end{split}
\]
The comparison principle implies \eqref{e:Ymin}.
\end{proof}
The lemma shows that on any finite time interval the characteristics will not leave $\O$ and in fact the image $Y(t,\supp \mu_0)$ will be compactly embedded in $\O$ and remain uniformly bounded a priori. Consequently, by extension, the system \eqref{e:Y} - \eqref{e:Th} is globally well-posed. By the push-forward transport, there is a global measure-valued solution to \eqref{e:opinion}.
\begin{theorem}\label{t:gwp}
Given any measure $\mu_0 \in \cP_0(\O)$ with \eqref{e:R0} there exists a unique measure-valued solution to \eqref{e:opinion} with initial condition $\mu_0$ and such that $\supp \mu_t \ss  \O$ remains bounded and bounded away from $\p \O$ uniformly for all times.
\end{theorem}

Let us now show continuity of the map $\mu_0 \to \mu_t$ in weak topology, which is the basis for justification of the mean-field limit.
\begin{lemma}\label{stability}
Let $\mu$ and $\nu$ be two measure-valued solutions to \eqref{e:opinion} with  $\mu_0,\nu_0$ satisfying \eqref{e:R0}. 
Then for any $t > 0 $ one has
\[
\cW_1(\mu_t,\nu_t) \leq C e^{ct}\cW_1(\mu_0,\nu_0),
\]
where $C,c>0$ depend on the initial condition and the parameters of the model.
\end{lemma}
\begin{proof}
Denote $ L^\infty := L^\infty (R_0)$.
Let us also denote by $Y$ the characteristics of $\mu$ and by $Z$ the characteristics of $\nu$.

In what follows, $C$ and $ c$ are constants which are varying line by line. By the definition of the Wasserstein distance, we have
\begin{align}\label{est-W1}
\cW_1(\mu_t,\nu_t) &=\sup_{\| \f \|_\Lip\leq 1}\left| \int_{\O} \f(y,\th)\dmu_t(y,\th) - \int_{\O} \f(y,\th)\dnu_t(y,\th)\right|\notag\\
& = \sup_{\| \f \|_\Lip\leq 1}\left| \int_{\O} \f(Y,\th)\dmu_0(y,\th) - \int_{\O} \f(Z,\th)\dnu_0(y,\th)\right| \notag\\
& = \sup_{\| \f \|_\Lip\leq 1}\left| \int_{\O} \f(Y,\th)\dmu_0(y,\th) - \int_{\O} \f(Y,\th)\dnu_0(y,\th)+ \int_{\O} [\f(Y,\th) -\f(Z,\th) ]\dnu_0(y,\th) \right| \notag\\
&\leq (1+\|\n Y\|_{\infty})\cW_1(\mu_0,\nu_0)+\int_{\O} |Y - Z |\dnu_0(y,\th) \notag\\
&\leq  ( 1+ \|\n Y\|_{\infty})\cW_1(\mu_0,\nu_0) + \| Y - Z  \|_{{\infty}}.
\end{align}

The proof reduces to the estimation of $\|\n Y\|_{\infty}$ and $ \| Y - Z  \|_{{\infty}}$.

Taking the gradient 
\[
\n Y = \left(\p_y Y,\p_\th Y \right)
\]
of \eqref{e:Y} we obtain
\[
\ddt \n Y = - \n Y +\th\n Y + (0, Y) - (p+1)Y^p \n Y.
\]
Evaluating at a point where $\|\n Y \|_{\infty}$ is achieved, 
\begin{equation}\label{ess-Y}
\ddt \|\n Y \|_{\infty} \leq - (1-\th)\|\n Y\|_{\infty} - (p+1)Y^p \|\n Y\|_{\infty} + \| Y\|_{\infty}.
\end{equation}
By \eqref{e:Ymax}, 
\[
\ddt \|\n Y \|_{\infty} \leq C \|\n Y\|_{\infty}  + C,
\]
and hence,
\begin{equation}\label{exp-gradY}
\|\n Y \|_{L^\infty} \leq C e^{c t}.
\end{equation}

Now let us compute the derivative of $ \| Y - Z\|_{\infty}$. We have
\[
\begin{split}
\ddt  (Y - Z ) &=
\int_{\O} (Y'-Y)\dmu_0(y',\th') -\int_{\O} (Z'-Z)\dnu_0(y',\th') \\
& \quad + (\th - Y^p)Y - (\th - Z^p) Z\\
&=\int_{\O} Y' \dmu_0(y',\th')-\int_{\O} Y' \dnu_0(y',\th') + \int_{\O} Y' \dnu_0(y',\th')-\int_{\O} Z'\dnu_0(y',\th') \\
&\quad + ( \th - 1)(Y - Z) - ( Y^{p+1} - Z^{p+1}).
\end{split}
\]
Evaluating at a point of maximum and noting that $Y^{p+1} - Z^{p+1} = (p+1) \tilde{Y}^p (Y- Z)$ for some $\tilde{Y}$  between $Y$ and $Z$ we obtain
\begin{equation}\label{}
\begin{split}
    \ddt  \| Y - Z \|_{\infty} &  \leq  \|\n Y\|_{\infty}\cW_1(\mu_0,\nu_0) + (| \th - 1| + 1) \| Y - Z \|_{\infty} -  (p+1) \tilde{Y}^p \| Y - Z \|_{\infty}  \\
    & \leq \|\n Y\|_{\infty}\cW_1(\mu_0,\nu_0) + C \| Y - Z \|_{\infty}.
\end{split}
\end{equation}
Combining with \eqref{exp-gradY}  and by Gr\"onwall's lemma, it implies that
\begin{equation}\label{exp-infty-dist-Y}
\| Y - Z \|_{{\infty}} \leq Ce^{ct}\cW_1(\mu_0,\nu_0) .
\end{equation}
where $c$ is a constant depending on $\s$ and the supports of $\mu_0, \nu_0$ with respect to $\th$. 
Therefore, plugging \eqref{exp-infty-dist-Y} and \eqref{exp-gradY} into \eqref{est-W1} we obtain
\begin{equation}
    \cW_1(\mu_t,\nu_t) \leq Ce^{ct} \cW_1(\mu_0,\nu_0)
    \end{equation}
which concludes the lemma.
\end{proof}

For any $N\in \N$, if $\{(y_i, \th_i)\}_{i= 1,\ldots, N}$ is a solution to the system \eqref{e:opiniondisc} with the initial conditions $y_i(0) = y_i^0, \th_i(0) = \th_i$, then 
\[
\mu^N_t :=  \frac1N \sum_{i=1}^N \d_{y_i(t)} \otimes \d_{\th_i},
\]
is a measure-valued solution to \eqref{e:opinion} with the initial condition
\[
\mu^N_0 =  \frac1N \sum_{i=1}^N \d_{y_i^0} \otimes \d_{\th_i}.
\]

So, if $\mu^N_0 \to \mu_0$ weakly, then by \lem{stability}, $\mu^N_t \to \mu_t$, for any $t>0$.  Which justifies the weak approximation by empirical measures.

This method can be used to give an alternative proof of global existence for \eqref{e:opinion} without the use of general characteristics $Y$ and simply based on the fact that the discrete system \eqref{e:opiniondisc} is globally well-posed. 

\begin{proof}[Another proof of \thm{t:gwp}]

Let us pick any weak$^*$-approximation of $\mu_0$ by empirical measures 
\[
\mu_0^N = \sum_{k= 1}^N m_k \d_{y^0_k}\otimes\d_{\th_k} \to \mu_0.
\]
Let 
\[
\mu_t^N := \sum_{k= 1}^N m_k \d_{y_k(t)}\otimes\d_{\th_k}.
\]
Since $\mu^N$ is a measure-valued solution to \eqref{e:opinion} with the initial data $\mu_0^N$ we apply Lemma \ref{stability} to get
\[
\cW_1(\mu_t^{N}, \mu_t^{M}) \leq Ce^T \cW_1(\mu_0^{N}, \mu_0^{M}), \quad \text{for } N,M>0, \ t\leq T.
\]
Hence $\{\mu_t^{N}\}_i$  is weakly$^*$-Cauchy, and consequently there is a limit $\mu_t^N \to \mu_t\in \cP_+(\O)$, and moreover
\begin{equation}\label{e:muNmu}
\cW_1(\mu_t^{N}, \mu_t) \leq C_T \cW_1(\mu_0^{N}, \mu_0), \quad \text{for } N>0, \ t\leq T.
\end{equation}
In particular, this implies weak$^*$-continuity of the map $t \to \mu_t$.

We will show that this $\mu$ is a measure-valued solution to \eqref{e:opinion} with the given initial $\mu_0$. 

 Because $\mu^N$ is a measure-valued solution, for any test function $\f \in C^{\infty}_0([0,T)\times \O)$,
\begin{equation}\label{e:approx-soln}
\int_{\O} \f(t,y,\th)\dmu^N_t(y,\th) =     \int_{\O}\f(0,y,\th) \dmu^N_0(y,\th) + \int_0^t \int_{\O} [\p_s\f +  u_s^N \p_y\f]\dmu^N_s(y,\th)\ds,
\end{equation}
where 
\[
u_s^N = \int_{\O} y' \dmu_s^N(y',\th')  - y + (\th - y^p)y: =  P^N(s) + F(y,\th).
\]
The linear terms all converge to the natural limits by weak convergence.  Since $R$ is a fixed continuous function  we also have
and 
\[
\int_0^t \int_{\O} F\p_y\f\dmu^N_s(y,\th)\ds\longrightarrow \int_0^t \int_{\O} F\p_y\f\dmu_s(y,\th)\ds \quad \text{ as } N\ra \infty
\]
Note that the moments $P^N(s)$ is just a sequence of numbers for which we have by \eqref{e:muNmu}
\[
| P^N(s) - P(s) | = \left| \int_{\O} y' ( \dmu_s^N(y',\th') - \dmu_s(y',\th')) \right| \leq \cW_1(\mu_s^N,\mu_s) \leq C_T \cW_1(\mu_0^N, \mu_0) \to 0.
\]
So, $P^N \to P$ uniformly on $[0,T)$. Consequently, 
\[
 \int_0^t \int_{\O} P^N(s)  \p_y\f \dmu^N_s(y,\th)\ds \to  \int_0^t \int_{\O} P(s)  \p_y\f \dmu_s(y,\th)\ds .
 \]
It follows that $\mu$ satisfies (ii). 
\end{proof}

\section{Existence and uniqueness of the mono-opinion state}
Let $\mu$ be a measure-valued solution to \eqref{e:opinion} with the initial $\mu_0$. Let
$\pi$ be its time-independent conviction marginal \eqref{e:pi}.

  Let us derive the equation for $\mu^\th$. By Definition \ref{ms-soln} and \eqref{slicing-ms},
for any $\f \in C^{\infty}_0([0,T)\times \O)$ and $0<t<T$ one has
\[
\begin{split}
\int_{\R_+}\int_{\R_+}  \f(t,y,\th)\dmu^\th_t(y)\dpi(\th) &=     \int_{\R_+}\int_{\R_+}   \f(0,y,\th) \dmu^\th_0(y)\dpi(\th) \\ &\quad \quad\quad + \int_0^t \int_{\R_+}\int_{\R_+}   [\p_s\f +  u_s \p_y\f]\dmu^\th_s(y)\dpi(\th)\ds.
\end{split}
\]
Note that 
\[
F(\mu_s)(y,\th) = \int_{\O}(z- y)\dmu_s(z,\eta) = \int_{\R_+}\int_{\R_+} (z-y)\dmu^\eta_s(z)\dpi(\eta).
\]
It implies that for each $\th$, the probability measure $\mu^\th$ is a measure-valued solution with the initial $\mu^\th_0$ to the equation
\begin{equation}\label{e:slicing-ms}
\p_t \mu^\th + \p_y \left[u \mu^\th \right]  = 0,
\end{equation}
where 
\[
u(t, y,\th) =  \int_{\O} (z-y)\dmu_t^\eta(z)\dpi(\eta) +  (\th - y^p)y. 
\]
Note that the family of equations are all coupled  through the velocity $u$, but otherwise represent transport of each individual slicing measure $\mu^\th$. The characteristics that transport $\mu^\th$, denoted $Y_\th$ are nothing but $Y_\th(t,y) = Y(t,y,\th)$ as defined by \eqref{e:Y}.  We will view them, however, as individual trajectories satisfying the coupled system
\begin{equation}\label{e:Yth}
    \ddt Y_{\th} = \int_{\R_+}\int_{\R_+}(Y'_{{\th'}} - Y_{\th} )\dmu^{\th '}_0 (y') \dpi(\th') + (\th - Y^p_{\th})Y_{\th}.
\end{equation}
In particular we will derive an individual comparison bound from below as an alternative to global \eqref{e:Ymin}.
\begin{lemma}\label{l:Ythmin}
    For any $\th \in [\thmin, \thmax]$ such that $\th > 1$ one has
    \begin{equation}\label{e:Ythmin}
        Y^p_\th(t,y) \geq 
        \frac{y^p ( \th - 1) e^{p(\th - 1) t}}{( \th - 1) +   y^p (e^{p(\th - 1) t} -1)  },\quad \forall t\geq 0, \ \forall y>0.
    \end{equation}    
\end{lemma}
\begin{proof}
To achieve \eqref{e:Ythmin} we decouple the system \eqref{e:Yth} by ignoring the entire coupling term
    \[
    \int_{\R_+}\int_{\R_+}Y'_{{\th'}} \dmu^{\th '}_0 (y') \dpi(\th') \geq 0.
    \]
  So,    \begin{equation}\label{Yth-prime}
\ddt Y^p_\th \geq  p\left( \th -1 -Y_\th^p\right) Y^p_\th.
    \end{equation}  
     The lemma follows from the comparison principle.
\end{proof}
Let us note that in principle the statement of the lemma holds for any $ \th - 1$, but it is most meaningful when the parameter is positive in view of the universal support from below for all characteristics \eqref{e:Ymin}.

\subsection{Mono-opinion state}
In the next step we will show that for each $\th \in \supp \pi$, the slicing measure $\mu^\th $ will converge to a Dirac measure in Wasserstein distance with different rates depending on $\th$.
\begin{lemma}\label{l:mono}
Let $\mu$ be the measure-valued solution to \eqref{e:opinion} satisfying \eqref{e:R0} and $\pi$ being the conviction marginal \eqref{e:pi}. Then there exists a function $g\in \Lip[\thmin,\thmax]$ such that\begin{equation}\label{exp-th}
\sup_{\th \in \supp \pi}\cW_1 (\mu_t^\th, \d_{g(\th)}) \leq  C e^{-c t}, \quad t>0,
\end{equation}
where $C,c>0$ depend only on $\mu_0$ and the parameters of the model.
\end{lemma}
\begin{proof}
Differentiating the characteristic equation \eqref{e:Yth} we obtain
\begin{equation}\label{e:Yy}
\p_t \p_y Y_{\th} = ( \th - 1) \p_y Y_{\th}  -  (p+1)Y_{\th}^p \p_y Y_{\th}.
\end{equation}
In what follows we denote $L^\infty = L^\infty(R_0)$. Evaluating at a point of maximum $y$ such that $(y,\th) \in R_0$, 
\begin{equation} \label{e:Yth-aux}
\ddt \| \p_y Y_{\th}\|_{\infty} =  ( \th - 1) \| \p_y Y_{\th}\|_{\infty} -  (p+1)Y_\th^p \| \p_y Y_{\th}\|_{\infty}.
\end{equation}

Let us first consider the stable case when $\th - 1 \leq \e_0$, with $\e_0>0$ to be determined later.  Using \eqref{e:Ymin} we find that $Y_\th^p \geq c_0$, which is determined only by the initial condition and the parameters of the model. Plugging in \eqref{e:Yth-aux}, we obtain
\begin{equation} \label{e:Yth-aux2}
\ddt \| \p_y Y_{\th}\|_{\infty} \leq   \e_0 \| \p_y Y_{\th}\|_{\infty} -  (p+1)c_0 \| \p_y Y_{\th}\|_{\infty} \leq - \e_0 \| \p_y Y_{\th}\|_{\infty} 
\end{equation}
by setting $\e_0 = \frac{ (p+1) c_0}{2}$. 

For the unstable case $ \th - 1 \geq \e_0$, we use \eqref{e:Ythmin}, which implies that 
\begin{equation}\label{e:Ythlow}
Y_\th^p \geq \th - 1  - c_1 e^{-c_2t},
\end{equation}
where $c_1,c_2>0$ also depends only on the initial condition and parameters of the model. Hence,
\[
\ddt \| \p_y Y_{\th}\|_{\infty} \leq  ( \th - 1) \| \p_y Y_{\th}\|_{\infty} -  (p+1)( \th - 1- c_1 e^{-c_2t})  \| \p_y Y_{\th}\|_{\infty} \leq [ - p \e_0 + c_1 e^{-c_2t} ] \| \p_y Y_{\th}\|_{\infty}.
\]

In either case we obtain, by Gr\"onwall's lemma, 
\begin{equation}\label{e:expY}
     \| \p_y Y_{\th}\|_{L^\infty} \leq c_3 e^{-c_4 t}.
\end{equation}
Consequently,
\begin{equation}\label{Yth-converge}
|Y_\th(y,t) - Y_\th(y',t)| \leq c_5 e^{-c_4 t}, \quad\quad \text{ for any } (y,\th), (y',\th) \in R_0.
\end{equation}
We can see that the characteristics are squeezing as $t$ approaches infinity. Since the trajectories are also precompact, for each $\th \in [\thmin,\thmax]$ there exists $g(\th)$ such that 
\begin{equation*}
\sup_{\ y\in  [\ymin,\ymax]} |Y_\th(y,t) - g(\th)| \leq c_5 e^{-c_4 t} .
\end{equation*}
We compute
\begin{align}
\cW_1 (\mu_t^\th, \d_{g(\th)}) &= \sup_{\| \f \|_\Lip\leq 1}\Big|\int_{\R_+}\f(y)\dmu^\th_t(y) - \int_{\R_+}\f(y)\d_{g(\th)}(y) \Big|\notag\\
& = \sup_{\| \f \|_\Lip\leq 1}\Big|\int_{\R_+}\f(Y_\th)\dmu^\th_0(y)-\f(g(\th))\Big|\notag\\
& = \sup_{\| \f \|_\Lip\leq 1}\Big|\int_{\R_+} (\f(Y_\th) -\f(g(\th))) \dmu^\th_0(y)\Big|\notag\\
&\leq \|Y_\th - g(\th)\|_{\infty}\notag.
\end{align}
The statement \eqref{exp-th} follows.

It remains to show that $g$ is a Lipschitz function on $[\thmin,\thmax]$. Indeed, computing the evolution of $\p_\th Y_\th$ we obtain
\[
\p_t \p_\th Y_\th = Y_\th + (\th  -1 - (p+1) Y_\th^p ) \p_\th Y_\th.
\]
Note that $Y_\th $ remains bounded on $R_0$, and the remainder of the equation has the same structure as in \eqref{e:Yy}. So, 
\[
\ddt \| \p_\th Y_\th \|_\infty \leq  c_1 + (- c_2 + c_3 e^{-c_4 t} )  \| \p_\th Y_\th \|_\infty.
\]
We obtain
\begin{equation}\label{ }
 \| \p_\th Y_\th \|_\infty < C.
\end{equation}
Consequently,
\[
| Y(y,\th,t) - Y(y,\th',t)| \leq C |\th - \th'|.
\]
Letting $t \to \infty$ we obtain
\[
|g(\th) - g(\th')| \leq C |\th - \th'|.
\]
This finishes the proof.
\end{proof}

\subsection{Uniqueness and stability}

The uniqueness of the limiting state follows from the lemma below and holds under either of the two conditions on parameters
\begin{equation}\label{e:!ass}
\thmin > \frac{p+1}{p} \quad \text{or} \quad \frac{\thmax}{\thmin} < p+1.
\end{equation}
Note that under the change \eqref{e:change} this translates into condition \eqref{e:!}.

\begin{lemma}\label{l:unique}
Let $\mu$ and $\tilde{\mu}$ be two solutions to \eqref{e:opinion} starting in a box $R_0$ and sharing the same conviction measure $\pi$. And suppose either of the assumptions \eqref{e:!ass} hold. Then for any $t\in [0,T)$ one has
\begin{equation}\label{ineq:stability}
\sup_{\th \in 
\supp \pi}\cW_1(\mu^\th_t,\tilde{\mu}^\th_t) \leq c_1 e^{-c_2 t}\sup_{\th \in 
\supp \pi}\cW_1(\mu^\th_0,\tilde{\mu}^\th_0) ,
\end{equation}
where $c_1,c_2>0$ depend on the initial data and parameters of the model.
\end{lemma}
\begin{proof}
In what follows $L^\infty: = L^\infty([\ymin,\ymax])$.  Denoting $\tilde{Y}_{\th}, \tilde{Y}'_{\th}$ the characteristics of $\tmu^\th$ starting from $y, y'$ respectively.
 For fixed $\th\in \supp \pi$,
\[
\begin{split}
    \cW_1(\mu^\th_t,\tmu^\th_t) &= \sup_{\| \f \|_\Lip\leq 1} \Big|\int_{\R_+}\f(y)\dmu^\th_t(y) - \int_{\R_+}\f(y)\dtmu^\th_t(y)\Big |\\
    & = \sup_{\| \f \|_\Lip\leq 1}\left| \int_{\R_+} \f(Y_{\th})\dmu^\th_0(y) - \int_{\R_+} \f(\tilde{Y}_{\th}) \dtmu^\th_0 (y) \right| \\
&= \sup_{\| \f \|_\Lip\leq 1}\left| \int_{\R_+} \f(Y_{\th})\dmu^\th_0(y) - \int_{\R_+} \f(Y_{\th})\dtmu^\th_0(y) + \int_{\R_+}[\f(Y_{\th})- \f(\tilde{Y}_{\th}) ]\dtmu^\th_0 (y) \right| \\
   &\leq \quad \|\p_y Y_{\th}\|_{L^\infty} \cW_1(\mu^\th_0,\tmu^\th_0) + \| Y_{\th} - \tilde{Y}_{\th} \|_{L^\infty}.
\end{split}
\]
We proved the uniform exponential contraction for $\|\p_y Y_{\th}\|_{L^\infty} $ in \eqref{e:expY}.  

Let us now focus on $\| Y_{\th} - \tilde{Y}_{\th} \|_{L^\infty}$.   We have
\begin{equation*}\label{}
\begin{split}
\ddt (Y_{\th} - \tilde{Y}_{\th} ) &= \int_{\R_+}\left[ \int_{\R_+}Y'_{{\th'}} \dmu^{\th'}_0(y') - \int_{\R_+}\tilde{Y}'_{{\th'}} \dtmu^{\th'}_0(y') \right] \dpi(\th') \\
&+ (\th-1) (Y_{\th} - \tilde{Y}_{\th}) - (Y_{\th}^{p+1} - \tilde{Y}_{\th}^{p+1})\\
& = \int_{\R_+} \Big[ \int_{\R_+} Y'_{{\th'}} (\dmu^{\th'}_0(y') - \dtmu_0^{\th'}(y')) + \int_{\R_+} (Y'_{{\th'}} - \tilde{Y}'_{{\th'}}) \dtmu^{\th'}_0 (y')\Big]\dpi(\th')\\
&+(\th-1) (Y_{\th} - \tilde{Y}_{\th}) -  (p+1)\hat{Y}_\th^p(Y_{\th} - \tilde{Y}_{\th}),
\end{split}
\end{equation*}
where $\hat{Y}_\th$ is between $Y_{\th} $ and $ \tilde{Y}_{\th}$. 
Denote
\[
\cD(t) = \sup_{\th \in [\thmin,\thmax]}  \| Y_{\th} - \tilde{Y}_{\th}\|_{L^\infty}.
\]
At a point of maximum we obtain using \eqref{e:expY},
\[
   \ddt \cD  \leq  c_3 e^{-c_4 t} \sup_{\th \in \supp \pi} \cW_1(\mu^{\th}_0, \tmu^{\th}_0)  +  \th \cD -  (p+1)\min\{Y_\th^p, \tilde{Y}_\th^p \}\cD.
\]
Using \eqref{e:Ythlow},
\[
\begin{split}
    \ddt \cD & \leq  c_3 e^{-c_4 t} \sup_{\th \in \supp \pi} \cW_1(\mu^{\th}_0, \tmu^{\th}_0)  +  \th \cD - (p+1)[  \th - 1  - c_1 e^{-c_2t}] \cD \\
    & = c_3 e^{-c_4 t} \sup_{\th \in \supp \pi} \cW_1(\mu^{\th}_0, \tmu^{\th}_0)  + [ p+1 - p  \th + c_1 e^{-c_2t}] \cD
\end{split}
\]
The result follows provided $ \thmin > \frac{p+1}{p}$. Alternatively, using the lower bound \eqref{e:Ymin},
\[
 \ddt \cD  \leq  c_3 e^{-c_4 t} \sup_{\th \in \supp \pi} \cW_1(\mu^{\th}_0, \tmu^{\th}_0)  + [ \thmax  -  (p+1) \thmin + c_1 e^{-c_2t}] \cD
\]
and the result follows provided $\frac{\thmax}{\thmin} < p+1$.
\end{proof}

Under the stability assumption \eqref{e:!ass} the limiting states are also stable with respect to perturbation of convictions. So, a small change even in the weak topology of conviction marginal $\pi$ results in a small change in the limiting mono-opinion state.
This can be proved via a minor modification of the argument above.

First, since we will be comparing slicing measures that are technically defined not  on the same set let us adopt a convention that if $\th\not\in \supp \pi$, then $\mu^\th = 0$.

\begin{lemma}\label{l:stability}
Let $\mu$ and $\tilde{\mu}$ be two measure-valued solutions to \eqref{e:opinion}  with the conviction marginals $\pi$ and $\tpi$, respectively, and parameters satisfying \eqref{e:!ass}. Then for any $t\in [0,T)$ one has
\begin{equation}\label{ineq:stability}
\sup_{\th \in [\thmin,\thmax] }\cW_1(\mu^\th_t,\tilde{\mu}^\th_t) \leq c_1 e^{-c_2 t}\sup_{\th \in 
 [\thmin,\thmax] }\cW_1(\mu^\th_0,\tilde{\mu}^\th_0) +c_3 e^{-c_4 t} + c_5 \cW_1(\pi,\tpi),
\end{equation}
where $c_i>0$  depend only on the initial condition and parameters of the model. 
\end{lemma}

By sending $t \to \infty$ and using that fact that 
\[
\sup_{\th \in [\thmin,\thmax] }   |g(\th)-\tilde{g}(\th) |   = \sup_{\th \in [\thmin,\thmax] } \cW_1( \d_{g(\th)}, \d_{\tilde{g}(\th)}),
\]
we obtain the statement \eqref{e:gLip} of \thm{t:main}.

\begin{proof}
We only need to focus on estimation of $\cD(t)$. We have
\begin{equation*}\label{}
\begin{split}
\ddt (Y_{\th} - \tilde{Y}_{\th} ) &= \int_{\R_+}\int_{\R_+}Y'_{{\th'}} \dmu^{\th'}_0(y')\dpi(\th') - \int_{\R_+}\int_{\R_+}\tilde{Y}'_{{\th'}} \dtmu^{\th'}_0(y')\dtpi(\th')  \\
& + (\th-1) (Y_{\th} - \tilde{Y}_{\th}) - (Y_{\th}^{p+1} - \tilde{Y}_{\th}^{p+1})\\
& = \int_{\R_+} \int_{\R_+} Y'_{{\th'}} \dmu^{\th'}_0(y') \dpi(\th') - \int_{\R_+} \int_{\R_+} \tilde{Y}'_{{\th'}} \dmu^{\th'}_0(y') \dpi(\th')\\
&+ \int_{\R_+} \int_{\R_+} \tilde{Y}'_{{\th'}} \dmu^{\th'}_0(y') \dpi(\th') - \int_{\R_+} \int_{\R_+} \tilde{Y}'_{{\th'}} \dtmu^{\th'}_0(y') \dpi(\th') \\
&+ \int_{\R_+} \int_{\R_+} \tilde{Y}'_{{\th'}} \dtmu^{\th'}_0(y') \dpi(\th') - \int_{\R_+}\int_{\R_+}\tilde{Y}'_{{\th'}} \dtmu^{\th'}_0(y')\dtpi(\th') \\
&+  (\th-1) (Y_{\th} - \tilde{Y}_{\th}) - (p+1)\hat{Y}_\th^p(Y_{\th} - \tilde{Y}_{\th}).
\end{split}
\end{equation*}
Hence,
\[
\ddt \cD \leq    c_3 e^{-c_4 t} \sup_{\th \in  [\thmin,\thmax]} \cW_1(\mu^{\th}_0, \tmu^{\th}_0) + \int_{\R^+} G(\th') [ \dpi - \dtpi] + \th \cD -  (p+1)\min\{Y_\th^p, \tilde{Y}_\th^p \}\cD,
\]
where 
\[
G(\th) := \int_{\R_+}\tilde{Y}_{{\th}}(y) \dtmu^{\th}_0(y) = \int_{\R_+}(\tilde{Y}_{{\th}}(y) - \tilde{g}(\th)) \dtmu^{\th}_0(y) + \tilde{g}(\th) .
\]
Since the first term is bounded exponentially, and $\tilde{g}\in \Lip$, we have
\[
\int_{\R^+} G(\th') [ \dpi - \dtpi] \leq c_1 e^{-c_2 t}  +  \|\tilde{g}\|_{\Lip} \cW_1(\pi,\tpi).
\]
Coming back to the $\cD$-equation and estimating the rest of the right hand side as previously we obtain
\[
\ddt \cD \leq    c_3 e^{-c_4 t} \sup_{\th \in  [\thmin,\thmax]} \cW_1(\mu^{\th}_0, \tmu^{\th}_0) +c_1 e^{-c_2 t} + c_5  \|\tilde{g}\|_{\Lip} \cW_1(\pi,\tpi)  - c_6 \cD.
\]
The result follows.
\end{proof}

\section{Properties of mono-opinion states} The results of the previous sections establish that for each conviction measure there is at least one (and in some cases only one) limiting distributions of opinions $g \in \Lip[\thmin,\thmax]$. Technically it makes material sense to only consider values of $g$ on the $\supp \pi$, but to study analytic properties of $g$ it will be convenience to make full use of its existence on the closed interval $[\thmin,\thmax]$.

We have the following equation for $g$:
\begin{equation}\label{e:steady-state}
\int_{\R_+} g(\eta) \dpi(\eta) +(\th-1) g(\th) - g^{p+1}(\th) = 0, \qquad \forall \th \in [\thmin,\thmax].
\end{equation}

Although it is difficult to find the function $g$ explicitly, solutions to  \eqref{e:steady-state} exhibit certain universal features. One instance where $g$ is computable is when $p=1$.  Indeed, let
\[
\a := \int_{\R_+} g(\eta) \dpi(\eta),
\]
then by \eqref{e:steady-state} we have
\[
g^2 +(1-\th) g - \a = 0.
\]
This second order equation always has a positive solution
\[
g =[ \th-1 + \sqrt{(1-\th)^2 + 4\a}]/2, 
\]
for any parameter $\a >0$. Note that this expression is still implicit as $\a$ depends on $g$. But whatever $\a$ is we can see in particular that $g$ is strictly increasing and convex.

Let us discuss these properties more systematically.

First, let us consider the extreme values
\[
\gmax = \max_{[\thmin,\thmax]} g(\th), \qquad \gmin = \min_{[\thmin,\thmax]} g(\th).
\]
Then clearly,
\begin{equation}\label{e:gminmax}
\thmin \leq \gmin^p, \quad \gmax^p \leq \thmax. 
\end{equation}

By \eqref{e:steady-state}, we also have that
\[
(\th-1) g(\th) - g^{p+1}(\th)  \leq 0, \qquad \forall \th \in [\thmin,\thmax].
\]
Thus, for each $\th \in [\thmin,\thmax]$ the following estimate holds true
\begin{equation}\label{g-rough-bound}
    g^p(\th) \geq \th - 1.
\end{equation}
A more refined estimate will be obtained next.
\begin{lemma}
  Let $g$ be a solution to the equation \eqref{e:steady-state}.  Then $g\in C^\infty([\thmin,\thmax])$, $g$ is strictly increasing on $[\thmin,\thmax]$, and for each $\th \in [\thmin,\thmax]$,
\begin{equation}\label{g-low-bound}
    g^p(\th) \geq \th + \pi( [\th,\infty)) - 1.
\end{equation} 
\end{lemma}
\begin{proof}
Since $g$ is Lipschitz we can conclude monotonicity from the sign of the derivative,
\begin{equation}\label{e:gprime}
g' = \dfrac{g}{1-\th + (p+1)g^p}.
\end{equation}
If $1 \geq  \th$, then using \eqref{e:gminmax}, it is clear that the denominator is positive, and so $g'>0$. If $1 < \th$  we have by the rough bound \eqref{g-rough-bound}
  \[
 1-\th + (p+1)g^p \geq p(\th - 1) > 0.
  \] 
This establishes monotonicity.  Also, since  the denominator of \eqref{e:gprime} is always positive, by bootstrapping this implies $g\in C^\infty([\thmin,\thmax])$. 

Combining monotonicity with the equation \eqref{e:steady-state} we obtain
\[
\int_{\{\eta \geq \th\}} g(\th) \dpi(\eta) - g(\th) +[\th- g^p(\th)]g(\th) \leq 0.
\]
Since $g(\th) \geq 0$ for all $\th \in [\thmin,\thmax]$ we must have
\[
\int_{\{\eta \geq \th\}} \dpi(\eta) - 1 +\th- g^p(\th) \leq 0.
\]
The estimate \eqref{g-low-bound} follows.
\end{proof}

Let us discuss convexity.   The second derivative of $g(\th)$ is given by
\[
g'' = \dfrac{ g'[1-\th +(p+1)g^p] -  g[-1 +  p(p+1)g^{p-1} g']}{[1-\th +(p+1)g^p]^2}
\]
and using \eqref{e:gprime} to replace $g'$ we obtain
\begin{equation}\label{g''}
g''= \dfrac{2(1-\th) g +(2+p-p^2)  g^{p+1}}{[1-\th +(p+1)g^p]^3}. 
\end{equation}
The denominator is always positive, and we note that in view of \eqref{g-rough-bound} the numerator is also positive regardless of the range of $\th$ provided $p\leq 1$. So, $g$ is globally convex in this case. 

In other cases, the convexity may change. In fact for $p=2$ we have
\[
g'' = \dfrac{2(1-\th) g}{[1-\th +3 g^2]^3}.
\]
So,  $\th = 1$ is an inflection point.

For $p> 2$, the solution has no more than one inflection point.  This can be seen by solving for $g'' = 0$ in \eqref{g''}. We have
\[
2(1- \th) = (p^2 - p - 2) g^p.
\]
The left hand side is a decreasing function and the right hand side is increasing for $p>2$. So, the two can meet at most at one point.

The exact value of $\a$ depends on $g$ and since the solution is in general not possible to compute explicitly we present in the figure below solutions to \eqref{e:steady-state} with several `passive' choices of $\a$ for illustration.

\begin{figure}[H]
\centering
\includegraphics[width=\textwidth, trim={0.5cm 12.5cm 0.5cm 2.5cm}, clip]{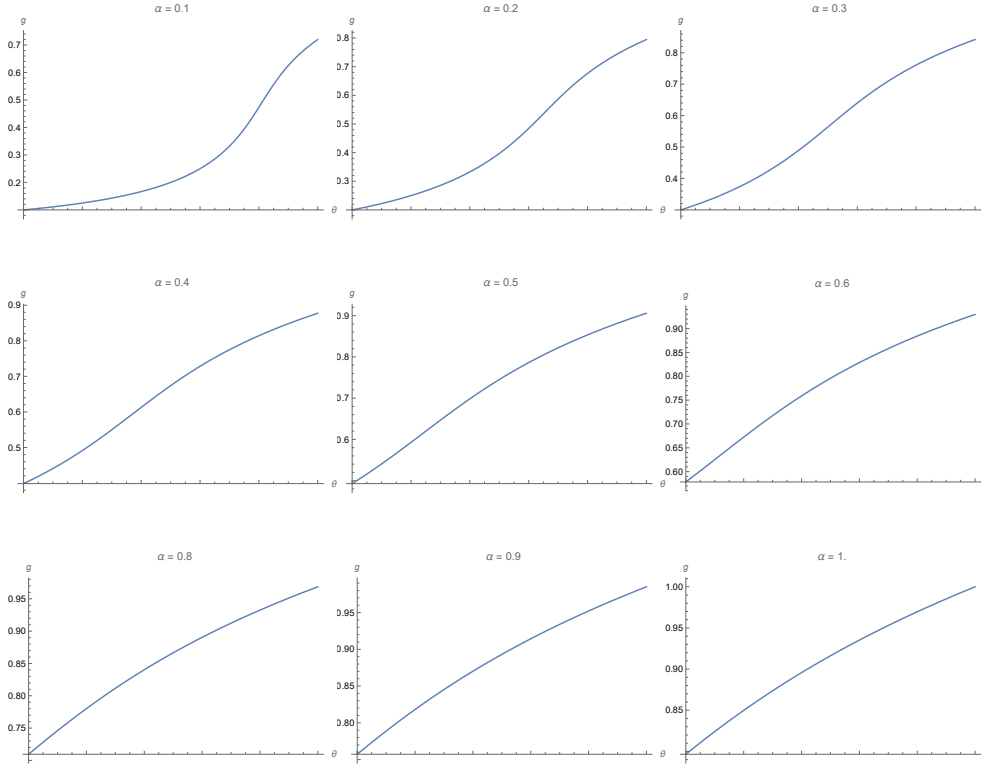}
\caption{The behavior of $g(\th)$ for the case $p=6$. Here $\th \in (0,1]$ and $\a$ change in $(0, 1]$ at discrete steps of $0.1$.}
\end{figure}


\newcommand{\etalchar}[1]{$^{#1}$}

\end{document}